\documentclass{aip-cp}

\usepackage{amsmath}
\usepackage{amsthm}
\usepackage{tikz,pgf}
\usepackage{setspace, comment}
\usepackage{hyperref}
\usepackage[top=1in, bottom=1.25in, left=1.25in, right=1.25in]{geometry}
\usepackage{enumitem}
\usepackage{abstract}
\usepackage[nointegrals]{wasysym}
\usepackage[utf8]{inputenc}
\usepackage[T1]{fontenc}

\newtheorem{definition}{Definition}
\newtheorem*{theorem*}{Theorem}
\newtheorem{theorem}{Theorem}

\newtheorem{lemma}{Lemma}
\newtheorem{remark}{Remark}
\newtheorem{proposition}{Proposition}
\providecommand{\abstract}{}

\renewcommand{\refname}{REFERENCES}

\newcommand*\circled[1]{\tikz[baseline=(char.base)]{\node[shape=circle,draw,inner sep=2pt] (char) {#1};}}

\title{A Generalised Nehari Manifold Method for a Class of Non-Linear Schr\"odinger Systems in $\mathbb{R}^3$}

\begin{document}
\author[aff1]{\textbf{TOMMASO CORTOPASSI} \corref{cor1}}
\author[aff1,aff2,aff3]{\textbf{VLADIMIR GEORGIEV}}
\eaddress{georgiev@dm.unipi.it}

\affil[aff1]{Dipartimento di Matematica, Universit\'a di Pisa,
Largo B. Pontecorvo 5, 56100 Pisa, Italy.}
\affil[aff2]{Faculty of Science and Engineering, Waseda University, 3-4-1, Okubo, Shinjuku-ku, Tokyo 169-8555, Japan.}
\affil[aff3]{IMI-BAS, Acad. Georgi Bonchev Str., Block 8, 1113 Sofia, Bulgaria.}
\corresp[cor1]{tommaso.cortopassi@sns.it}
\maketitle
\begin{abstract}
   \noindent  We study the existence of positive solutions of a particular elliptic system in $\mathbb{R}^3$ composed of two coupled non linear stationary Schr\"odinger equations (NLSEs), that is $-\epsilon^2 \Delta u + V(x) u= h_v(u,v), - \epsilon^2 \Delta v + V(x) v=h_u (u,v)$. Under certain hypotheses on the potential $V$ and the non linearity $h$, we manage to prove that there exists a solution $(u_\epsilon,v_\epsilon)$ that decays exponentially with respect to local minima points of the potential and whose energy tends to concentrate around these points, as $\epsilon \to 0$. We also estimate this energy in terms of particular ground state energies. This work follows closely what is done in \cite{ramos2008solutions}, although here we consider a more general non linearity and we restrict ourselves to the case where the domain is $\mathbb{R}^3$.
\end{abstract}

\begin{section}{\textbf{INTRODUCTION}}

We will study the following system of two non linear stationary Schr\"odinger equations \footnote{Reprinted with permission from "AIP Conference Proceedings" \textbf{2459}, 030003 (2022), ; \href{https://doi.org/10.1063/5.0084041}{; https://doi.org/10.1063/5.0084041}}:

\begin{equation}\label{main system}
\begin{cases}
- \epsilon ^2 \Delta u + V(x) u = h_v (u,v)\\
- \epsilon^2 \Delta v + V(x) v = h_u (u,v)\\
u,v \in H^1 (\mathbb{R}^3)\\
u,v >0 \text{ in } \mathbb{R}^3
\end{cases}
\end{equation}
with $\epsilon >0$. We will denote with $h_u$ and $h_v$ the partial derivatives of $h$ with respect to the first and second variable, respectively. Relevant hypotheses on $h$ and $V$ are:

\begin{itemize}
    \item [h1)] $h(s,t),h_u(s,t),h_v(s,t),h_{uu}(s,t), h_{uv}(s,t),h_{vv}(s,t) \geq 0$ for every $s,t$, and they are 0 if either $s$ or $t$ is less or equal than 0;
    
    \item[h2)]
     There exists $\epsilon' >0$ such that 
     
     \begin{equation}
         \liminf_{s \to + \infty} \left( \inf_{ t \geq \epsilon'} \frac{h_u (s,t)}{s^{p+q-1}} \right)
    \geq C
     \end{equation}

    and 
    
    \begin{equation}
        \liminf_{t \to + \infty} \left( \inf_{ s \geq \epsilon'} \frac{h_v (s,t)}{t^{p+q-1}} \right)
    \geq C
    \end{equation}
     for some $p, q \in (2,3)$ and $C>0$;
     
    
    \item[h3)] It exists $\delta' >0$ such that $0 < (1+ \delta') \begin{bmatrix} \frac{h_u (s,t)}{s} &  0 \\ 0 & \frac{h_v (s,t)}  {t}
    \end{bmatrix}  \leq  \begin{bmatrix} h_{uu}(s,t) & h_{uv} (s,t) \\ h_{uv} (s,t) & h_{vv} (s,t) \end{bmatrix}$ for all $s,t >0$, in the sense of scalar products;

    \item[h4)] There exists a constant $C>0$ and $2< p,q < 3$ such that: \begin{align}
    &|h_{uu} (s,t)| + |h_{vv}(s,t)| + |h_{uv}(s,t)|\leq C(|s|^2 +|t|^2 + |s|^{p+q-2} + |t|^{p+q-2});\\
    &|h_u(s,t)| \leq C|s|(|s|^2 +|t|^2 + |s|^{p+q-2} + |t|^{p+q-2});\\
    &|h_v(s,t)| \leq C|t|(|s|^2 +|t|^2 + |s|^{p+q-2} + |t|^{p+q-2});
    \end{align}
    for every $s,t > 0$;

    \item[h5)]  
    There exists $\delta'' >0$  such that  
    
    \begin{equation}
        h_u(s,t)s + h_v(s,t)t - 2 h(s,t) \geq \delta'' ( h_u(s,t)s + h_v(s,t)t)
    \end{equation}
    
    for $s,t > 0$;
    
    \item[h6)]  For every $\mu >0$ there exists $C_\mu >0$ such that:
    \begin{align}
     &|h_u (s,t) t| +|h_v(s,t)s| \leq \mu (s^2 + t^2) + C_\mu (s^6 + t^6);\\
     &|h_u (s,t) t| +|h_v(s,t)s| \leq \mu (s^2 + t^2) + C_\mu (h_u (s,t)s + h_v (s,t)t).
    \end{align}
    
    Moreover in the first inequality we can also assume that $C_\mu \to 0$ as $\mu \to + \infty;$ 
    
    \item[V1)] The function $V$ is locally H\"older continuous, it holds
    
    \begin{equation}
        \alpha \coloneqq \inf_{\mathbb{R}^3} V >0
    \end{equation}
    
    and $V$ belongs to the reverse H\"older class $RH_\infty$ (see \cite{shen1995p} for the precise definitions, but for instance polynomials belong to $RH_\infty$);
    
    \item[V2)] There exist bounded domains $\Lambda_i$ mutually disjoint with $i= 1, \dots, k$ such that $\inf_{\Lambda_i} V < \inf_{\partial \Lambda_i} V$. That is, $V$ admits $k$ strict local minima points $x_1, \dots, x_k$.
\end{itemize}

\begin{remark}
Examples of functions $h$ and $V$ which satisfy the aforementioned hypotheses may be

\begin{equation}\label{definition of h}
h(s,t)= \begin{cases}
s^{p+q} + t^{p+q} + s^p t^q \text{ if }  s,t > 0 \\

0 \text{ otherwise}

\end{cases}
\end{equation}
and 

\begin{equation}
    V(x)= (1 + |x|)^2
\end{equation}
with $p, q \in (2,3)$. Although here $h$ is not even continuous along the axes, it can easily be seen that we can multiply it by appropriate cutoff functions which cut along the axes so that it becomes $C^\infty$.
\end{remark}

\noindent The space we will be working with is

\begin{equation}
    H\coloneqq \left\{ u \in H^1  (\mathbb{R}^3)  | \int V(x) u^2 < +\infty \right\},
\end{equation}

\noindent which is a Hilbert space endowed with the scalar product

\begin{equation}
    \langle u, v \rangle_H \coloneqq \int \{ \langle \nabla u, \nabla v \rangle + Vuv\}dx.
\end{equation} 
Thanks to the fact that $\inf_{x \in \mathbb{R}^3} V(x)= \alpha >0$, it is trivial to check that $H \subseteq H^1 (\mathbb{R}^3)$ continuously. The approach we will use to study this problem is that of a "generalised Nehari manifold", and it will follow closely what has been done by Ramos and Tavares in \cite{ramos2008solutions}, but the novelty in our work is that the non linearity in \eqref{main system} is more general. We can see that system \eqref{main system} admits a variational characterization as the Euler-Lagrange equation of the functional 

\begin{equation}
    I_\epsilon (u,v)= \int_{\mathbb{R}^3} \epsilon ^2 \langle \nabla u, \nabla v \rangle + V(x) uv - h(u,v) dx.
\end{equation}

\noindent This functional is indefinite even in its quadratic part and it can easily be seen that trying to use the classical Nehari method does not work. The key ideas of our approach are:

\begin{itemize}
    \item Decompose the space $H \times H$ as $H^+ \oplus H^-$, with 
    
    $$H^\pm \coloneqq \{(\phi, \pm \phi) \mid \phi \in H\}.$$
    
    If restricted to $H^-$, we can see that $I_\epsilon$ is definite negative.
    
    \item Modify the non linearity $h(s,t)$ with a particular $h(x,(s,t))$ depending on $x$ in order to gain more control on its behavior. We can see a posteriori and thanks to an argument which ultimately relies on the maximum principle that a solution to the modified system actually solves \eqref{main system}. 
\end{itemize}

The main result is the following:

\begin{theorem}\textbf{Main Theorem}\label{main theorem}

\noindent Under hypotheses V1), V2), h1)-h6) there exists $\epsilon_0 >0$ such that for every $0 < \epsilon < \epsilon_0$ the system (\ref{main system}) admits classical positive solutions $u_\epsilon, v_\epsilon \in C^2(\mathbb{R}^3) \cap H^2 (\mathbb{R}^3)$ with locally h\"olderian second derivatives, such that:
\begin{enumerate}[label=(\roman*)]
    \item there exist $x_{i, \epsilon} \in \Lambda_i$ for every $i=1, \dots, k$, local maximum points for both $u_\epsilon$ and $v_\epsilon$;
    \item $u_\epsilon (x_{i, \epsilon}), v_\epsilon (x_{i, \epsilon}) \geq b >0$ and $V(x_{i, \epsilon}) \to V(x_i)= \inf_{\Lambda_i} V $ as $\epsilon \to 0$;
    \item $u_\epsilon (x), v_\epsilon (x) \leq \gamma e^{-\frac{\beta}{\epsilon} |x- x_{i, \epsilon}|}, \; \forall x \in \mathbb{R}^3 \setminus \cup_{j \neq i} \Lambda_j$
    \end{enumerate}
    for every $i=1, \dots , k$ and for some positive constants $b,\gamma, \beta$. The uniqueness of local maximum points holds in the following sense:
    \begin{enumerate}[label=(\roman*)]
    \setcounter{enumi}{3}
    \item if either $u_\epsilon$ or $v_\epsilon$ has a local maximum point at some $z_\epsilon \neq x_{i, \epsilon}$ for every $i=1, \dots, k$, then it holds $\lim_{\epsilon \to 0} u_\epsilon (z_\epsilon)= \lim_{\epsilon \to 0} v_\epsilon (z_\epsilon)=0.$
    \end{enumerate}
    
    The solution $(u_\epsilon, v_\epsilon)$ has its energy concentrated near $x_1, \dots , x_k$. Consider the problems
    
    \begin{equation}\label{ground state system in main theorem}
    \begin{cases}
    - \Delta u + V(x_i) u= h_v (u,v)\\
    - \Delta v + V(x_i) v= h_u (u,v)\\
    u,v \in H^1 (\mathbb{R}^3)\\
    u,v >0
    \end{cases}
    \end{equation}
    
    for $i=1, \dots k$ and let $c_i$ be their ground state energy levels, that is
    
    \begin{equation}
        c_i \coloneqq \inf \{ I_{V(x_i)} (u,v) | u \neq 0, v \neq 0 \text{ and } (u,v) \text{ solves } (\ref{ground state system in main theorem}) \},
    \end{equation}
    
    with $I_{V(x_i)}(u,v)= \int_{\mathbb{R}^3} \langle \nabla u, \nabla v \rangle + V(x_i) uv- \int_{\mathbb{R}^3} h(u,v)$ the functional associated to (\ref{ground state system in main theorem}). If we define 
    
    \begin{equation}
        I_\epsilon (u,v)\coloneqq \int_{\mathbb{R}^3} \epsilon ^2 \langle \nabla u, \nabla v \rangle + V(x) uv - \int_{\mathbb{R}^3} h(u,v)
    \end{equation}
    
    and 
    
    \begin{equation}
        I_\epsilon ^i (u,v) \coloneqq \int_{\Lambda_i} \epsilon ^2 \langle \nabla u, \nabla v \rangle + V(x) uv - \int_{\Lambda_i} h(u,v)
    \end{equation}
    
    then it holds that
    
    \begin{equation}
        I_\epsilon ^i (u_\epsilon, v_\epsilon) = \epsilon ^3 (c_i + o_\epsilon (1))
    \end{equation}
    
    and 
    
    \begin{equation}
        I_\epsilon (u_\epsilon, v_\epsilon)= \epsilon ^3 \left( \sum_{i=1} ^k c_i + o_\epsilon (1) \right)
    \end{equation}
    with $o_\epsilon (1) \to 0$ as $\epsilon \to 0$.
    \end{theorem}

For more details, we refer the reader to \cite{NehariMethodThesis}.
\end{section}

\begin{section}{\textbf{TECHNICAL LEMMAS}}

Before proving Theorem \ref{main theorem}, we will have to prove some technical lemmas. The proofs will be sketchy and sometimes we will directly refer the reader to \cite{ramos2008solutions}.

\begin{lemma}\label{lemma1}

Let $\phi \in H$, $W \in C^0 (\mathbb{R}^2)$ such that $W(s,t)=0$ if $s \leq 0$ or $t \leq 0$, and 

\begin{equation}\label{W estimate}
|W(s,t)| \leq C(|s|^2 +|t|^2 +|s|^{p+q-2} + |t|^{p+q-2})
\end{equation}

\noindent for some positive constant $C$ and $p,q \in (2,3)$. 
Then, fixing $(u,v) \in H \times H$, the map from $H$ to $H'$ sending $\phi$ in $W(u,v) \phi$ is compact, where

\begin{align}
    W(u,v)\phi:& H \longrightarrow \mathbb{R} \nonumber\\
    & \psi \mapsto \int W(u,v) \phi \psi.
\end{align}
\end{lemma}
\begin{proof}

We can assume without loss of generality that $\{\phi_n\}_{n \in \mathbb{N}} \subset H$ with $||\phi_n||_H =1 \; \forall n \in \mathbb{N}$. By taking a subsequence, we assume that $\phi_n \rightharpoonup 0$. We have to prove that $W(u,v)\phi_n \to 0$ in $H'$. To prove this, fix $\epsilon>0$ and choose $B_R $ a ball centered in the origin with $R=R(\epsilon)$ such that

\begin{equation}\label{little r norm}
||u||_{L^r (B_R ^c)} , ||v||_{L^r (B_R ^c)} \leq \epsilon  \text{ for every } r \in [2,6].
\end{equation}

\noindent Then 

\begin{equation}
    \int W(u,v) \phi_n \psi= \overbrace{\int_{B_R} W(u,v) \phi_n \psi}^{\circled{1}}+ \overbrace{\int_{B_R ^c} W (u,v) \phi_n \psi}^{\circled{2}},
\end{equation}

\noindent and we can conclude that $||W(u,v) \phi_n||_{H'} \to 0$ by using Sobolev compact embedding in $\circled{1}$ and H\"older inequality along with \eqref{little r norm} in $\circled{2}$.
\end{proof}

\begin{lemma}\label{lemma2}
Consider the decomposition $H \times H= H^{+} \oplus H^{-}$, where 

\begin{equation}
    H^{\pm}= \{(\phi, \pm \phi)|\phi \in H\}.
\end{equation}

For every $(u,v) \in H \times H$ it is possible to define $(\Psi_{u,v}, - \Psi_{u,v}) \in H^-$ as the unique minimum point of

\begin{align}
    F: &H^- \longrightarrow \mathbb{R} \nonumber\\
    &(\phi,- \phi) \mapsto -I((u,v)+ (\phi, -\phi)).
\end{align}

Moreover, the map 

\begin{align}\Theta: &H \times H \to H \nonumber\\ 
&(u,v) \mapsto \Theta(u,v)= \Psi_{u,v}
\end{align}

is $C^1$.
\end{lemma}
\begin{proof}

It is easy to see that, for every fixed $(u,v) \in H \times H$, $F$ admits a unique global minimum point. To prove that $\Theta$ is of class $C^1$ we use the implicit function theorem. Consider:

\begin{align}
G: &(H \times H) \times H^- \longrightarrow (H^-)' \nonumber\\
&((u,v),(\phi,-\phi)) \mapsto I' ((u,v)+ (\phi, - \phi))=-F'(\phi).
\end{align}

\noindent Under our hypotheses, the map $G$ is $C^1$, because calculating $G'$ we get

\begin{align}
    G'((u,v), \phi)((f,g), \psi)(\rho,-\rho)&= - 2\langle \psi, \rho \rangle_H + \langle g-f, \rho \rangle_H - \nonumber \\
    & -\int (f+\psi, g-\psi)H_h (u + \phi, v-\phi) (\rho, -\rho)^T.
\end{align}

\noindent We have already verified that $G((u,v), \Psi_{u,v})=0$ and the derivative of $G$ with respect to the variable $\phi$ only, evaluated in $((u,v), \Psi_{u,v})$ is given by:

\begin{align}
    \frac{\partial G}{\partial \phi}((u,v), \Psi_{u,v}) (\psi, -\psi) (\rho, -\rho)= -2 \langle \psi, \rho \rangle_H  - \int (\psi, -\psi)H_h (P) (\rho, -\rho)^T.
\end{align}
That is, $\frac{\partial G}{\partial \phi}((u,v), \Psi_{u,v}) (\psi, -\psi)= [-2 Id - H_h (P)](\psi, -\psi) \in (H^-)'$, where

\begin{itemize}
    \item $\underbrace{-2(\psi,-\psi)}_{\in (H^-)'}(\rho,-\rho)= -2 \langle \psi , \rho \rangle_H$,
    
    \item $\underbrace{H_h (P)(\psi, -\psi)}_{\in (H^-)'}(\rho, -\rho)= \int (\psi, -\psi) H_h (P) (\rho, -\rho).$
\end{itemize}

\noindent To prove that $\frac{\partial G}{\partial \phi}((u,v), \Psi_{u,v})=[-2 Id - H_h (P)](\psi, -\psi) \in (H^-)'$ is an isomorphism we can use Fredholm alternative along with Lemma \ref{lemma1}.
\end{proof}

\begin{lemma}\label{lemma6}
Let $(u_n, v_n)$ be a Palais-Smale sequence for the functional $I$, that is a sequence such that

\begin{itemize}
    \item $0 < \liminf_{n \to \infty} I(u_n,v_n) \leq \limsup_{n \to \infty} I(u_n,v_n) < +\infty;$
    \item $\mu_n \coloneqq ||I'(u_n,v_n)||_{(H \times H)'} \to 0$ as $n \to +\infty.$
\end{itemize}

Then $\sup_{s \geq 0} \{I(s(u_n,v_n)+ (\phi, -\phi)) | \phi \in H \}= I(u_n, v_n) + O(\mu_n ^2).$
\end{lemma}

\begin{proof}

See \cite{ramos2008solutions}.
\end{proof}

\begin{remark}\label{remark Fatou ground state}
Under the assumptions of Lemma \ref{lemma6}, we also have that 

\begin{equation}
    I_\lambda (u_n, v_n) \geq c(\lambda) + o_n(1) \text{ as } n \to \infty.
\end{equation}

\noindent Indeed up to a subsequence, $u_n \rightharpoonup u$ and $v_n \rightharpoonup v$ in $H^1 (\mathbb{R}^3)$. Up to another subsequence, the convergence can be considered pointwise almost everywhere. It is enough to consider bigger and bigger balls to have, inside them, a subsequence converging strongly in $L^2$ and so almost everywhere (eventually taking yet another subsequence). A simple diagonal argument gives us pointwise convergence almost everywhere to $(u,v)$. We have that $I_\lambda ' (u,v)=0$, which can be proved by checking $C_0 ^\infty$ functions only (which is not restrictive because they are dense in $H^1 (\mathbb{R}^3)$) and using Rellich-Kondrachov theorem in their support.
Using Fatou lemma and condition h5) we get

\begin{align}\label{liminf  estimate of ground state}
&2I_\lambda (u,v)= 2 I_\lambda (u,v) - \overbrace{I_\lambda ' (u,v)(u,v)}^{=0}= \int \overbrace{\langle \nabla h(u,v), (u,v) \rangle - 2 h(u,v)}^{\geq 0} = \nonumber\\
& \int \liminf_{n \to \infty} (\langle \nabla h(u_n,v_n), (u_n,v_n) \rangle - 2 h(u_n,v_n)) \leq \liminf_{n \to \infty} (2I_\lambda (u_n, v_n) - I_\lambda ' (u_n, v_n)(u_n, v_n))= \nonumber\\ &\liminf_{n \to \infty} 2 I_\lambda (u_n, v_n). 
\end{align}

\end{remark}

\noindent Thanks to Lemma \ref{lemma6}, it is possible to apply the argument used in Lemma 3.1 of \cite{ramos2006concentration} to have that the map $\lambda \mapsto c(\lambda)$ is increasing from $\mathbb{R}^+$ to $\mathbb{R}^+$. That is, at bigger values of $\lambda$ in (\ref{ground state system}) correspond bigger ground state energy levels. We only state the next Lemma:

\begin{lemma}\textbf{(see Lemma 3.1 in \cite{ramos2006concentration})}\label{ground state monotonicity with lambda}

\noindent For any constant $\lambda >0$, consider the ground state critical level $c(\lambda)$ defined as $c(\lambda) \coloneqq \inf \{I_\lambda (u,v)|u \neq 0, v \neq 0 \text{ and } I_\lambda ' (u,v)=0 \}$ associated to the problem

\begin{equation}
\begin{cases}
-\Delta u + \lambda u= h_v (u,v)\\
- \Delta v + \lambda v =h_u (u,v)\\
u, v \in H^1 (\mathbb{R}^3).
\end{cases}
\end{equation}

\noindent The map $\lambda \mapsto c(\lambda)$ is continuous and increasing.
\end{lemma}

\begin{lemma}\label{regularity of solutions}

If $u,v$ are solutions of (\ref{main system}), then $u, v \in H^2(\mathbb{R}^3) \cap C^{2} (\mathbb{R}^3)$ with locally h\"olderian second derivatives, and $u,v$ are strictly positive.
\end{lemma}

\begin{proof}

We can suppose without loss of generality that $\epsilon =1$. We know that

\begin{equation}
    - \Delta u + V(x) u= h_v (u,v) \lesssim |v|(|u|^2 +|v|^2+ |v|^{p+q-2} + |u|^{p+q-2}).
\end{equation}

The right hand side belongs to $L^{6/ (p+q-1)}$, and using Corollary 0.9 of \cite{shen1995p} we know that $\nabla^2 u \in L^{6/ (p+q-1)}$. This means that $ u \in W^{1, 6/(p+q-3)}$, where $\frac{6}{p+q-3}= \left( \frac{6}{p+q-1}\right)^*$, that is the Sobolev critical exponent associated to $6/(p+q-1)$. It is clear that repeating our argument with bigger exponents we will eventually get to a point where $u \in H^2$ and $u \in C^{0, \alpha}$ thanks to Sobolev embedding. The same holds for $v$. The positivity of $(u,v)$ and the fact that $u(x),v(x) \to 0$ as $|x| \to + \infty$ can easily be derived the maximum principle. Since $h(u,v)$ is h\"olderian and $V$ is locally H\"older too, we can obtain $C^2$ regularity with locally h\"olderian second derivatives using Schauder theory (see \cite{ambrosio2019lectures}, \cite{evans1998partial} or \cite{gilbarg2015elliptic}).
\end{proof}
\end{section}

\begin{section}{\textbf{GENERALISED NEHARI MANIFOLD}}

From now on, for each $i=1, \dots, k$ we fix mutually disjoint open sets $\Lambda_i '$ and $\tilde{\Lambda}_i$ such that $\Lambda_i \Subset \Lambda_i ' \Subset \tilde{\Lambda}_i$ and cutoff functions

\begin{equation}
    \phi_i (x)= \begin{cases}
        1 \text{ if } x \in \Lambda_i\\
        0 \text{ if } x \in \mathbb{R}^3 \setminus \Lambda_i ' .
                \end{cases}
\end{equation}

\noindent We also denote 

\begin{equation}
    \Lambda \coloneqq \cup_i \Lambda_i; \; \Lambda' \coloneqq \cup_i \Lambda_i ' ; \; \tilde{\Lambda} \coloneqq \cup_i \tilde{\Lambda}_i.
\end{equation}

\noindent We fix a sufficiently small $a >0$ (we will see later how small we need it to be) and we define the following modification of $h$, using polar coordinates for the sake of simplicity:

\begin{equation}
    \tilde{h}(\rho, \theta)=
\begin{cases}
    h(\rho,\theta) \text{ if } \rho \leq a \\
    h(a, \theta)+ h_\rho (a, \theta) (\rho - a) +  \frac{h_{\rho \rho} (a, \theta)}{2} (\rho- a)^2 \text{ if } \rho > a.
\end{cases}
\end{equation}

\noindent This modification is still $C^2$ and it is a sort of quadratic prolongation of $h$ outside a ball $B_a$. Then we define

\begin{equation}
    h(x,(u,v)) \coloneqq \chi_\Lambda (x) h(u,v) + (1- \chi_\Lambda (x)) \tilde{h}(u,v)
\end{equation}
 
\noindent and the corresponding functional will be:

\begin{equation}
    J_\epsilon (u,v) \coloneqq \int \lbrace \epsilon^2 \langle \nabla u, \nabla v \rangle + V(x) uv \rbrace - \int h(x,(u,v))= \langle u,v \rangle_{\epsilon} - \int h(x,(u,v)).
\end{equation}
 
\noindent We are now considering the scalar product

\begin{equation}
    \langle u, v \rangle_\epsilon \coloneqq \int \epsilon ^2 \langle \nabla u, \nabla v \rangle + V(x) uv
\end{equation}

\noindent and we also define

\begin{equation}
    ||(u,v)||_\epsilon ^2 \coloneqq ||u||_\epsilon ^2 + ||v||_\epsilon ^2 = \int \epsilon^2 |\nabla u|^2 + V(x) u^2 + \int \epsilon ^2 |\nabla v|^2 + V(x) v^2.
\end{equation}

\noindent Properties h1)-h7) are unaffected by this modification, but we gain a control outside of $\Lambda$, that is it holds the property

\begin{itemize}
    \item[hm)]For some $\delta= \delta(a)>0$, we have 
    
    \begin{equation}
        |\nabla h(x,(s,t))| \leq \delta \sqrt{s^2 + t^2} \; \forall x \in \mathbb{R}^3 \setminus \Lambda \; \forall u,v >0;
    \end{equation}
    \end{itemize}

\noindent Critical points of $J_\epsilon$ do not satisfy (at least, not a priori) system (\ref{main system}), but they will now satisfy the modified system

\begin{equation}\label{modified system}
\begin{cases}
    - \epsilon^2 \Delta u + V(x)u= h_v (x,(u,v)) \\
- \epsilon^2 \Delta v + V(x)v= h_u (x,(u,v))\\
u, v \in H^{1}(\mathbb{R}^3).
\end{cases}
\end{equation}

\noindent This modification is of crucial importance in our method, indeed we change $h$ in such a way that it remains the same if $(u,v)$ is either near a minimum point for $V$ or small enough in norm, while we bound $h$ with a quadratic growth when $(u,v)$ is far from minima points of $V$ and with a norm greater than a fixed threshold $a$.
This will make us able to find a solution for (\ref{modified system}). However, a posteriori, we can say thanks to an argument which relies basically on the maximum principle that, outside of $\Lambda$, $||(u_\epsilon,v_\epsilon)||$ is never greater than $a$, at least for a sufficiently small $\epsilon$. This means that the modification never occurs and $(u_\epsilon, v_\epsilon)$ actually satisfies (\ref{main system}).\newline

\noindent We are now ready to define the generalised Nehari manifold.

\begin{definition}\textbf{Generalised Nehari manifold}

\noindent We define the generalised Nehari manifold $N_\epsilon$ as the set of functions $(u,v) \in H \times H$ such that for every $i \in \{ 1, \dots, k\}$:

\begin{enumerate}
\item $J_\epsilon ' (u,v)(\phi, -\phi)=0 \; \forall \phi \in H$;
\item $J_\epsilon ' (u,v)(u \phi_i, v \phi_i)=0 \; \forall i=1, \dots, k$ ; 
\item $\int_{\Lambda_i} (u^2 + v^2) > \epsilon ^4 \; \forall i=1, \dots, k.$
\end{enumerate}
with $\phi_1, \dots, \phi_k$ smooth cutoff functions near $x_1, \dots, x_k$.
\end{definition}

\noindent Before going on, we must check that $N_\epsilon \neq \emptyset$, which is not a trivial fact. we will make use of ground state solutions whose energy levels correspond to local minima for $V$.

\begin{definition}\label{ground state definition}\textbf{Ground state solutions and ground state energy level}

\noindent For a given $\lambda >0$ we will denote with $I_\lambda$ the energy functional associated with the problem

\begin{equation}\label{ground state system}
\begin{cases}
  - \Delta u + \lambda u = h_v(u,v)\\
  - \Delta v + \lambda v= h_u (u,v)\\
  u, v \in H^1 (\mathbb{R}^3)
\end{cases}
\end{equation}

\noindent and by $c(\lambda)$ the associated ground state energy level, that we define as  

\begin{equation}
    c(\lambda) \coloneqq \inf \{I_\lambda (u,v)| u \neq 0, v \neq 0 \text{ and } I_\lambda ' (u,v)=0 \}.
\end{equation}

\noindent If $(u,v)$ is a solution of (\ref{ground state system}) and $I_\lambda (u,v)= c(\lambda)$, then we say that $(u,v)$ is a ground state solution.
\end{definition}

\noindent We already know (see Lemma \ref{lemma2}) that for every $(u,v) \in H \times H$ there exists $\Psi_{u,v} \in H$ so that $J_\epsilon ' (u+ \Psi_{u,v},v-\Psi_{u,v})|_{H^-}=0$. Most of the proof of Proposition \ref{prop3.1} will indeed deal with the second condition in the definition of $N_\epsilon$, that is to find $(\bar{u}, \bar{v})$ that satisfy $J_\epsilon ' (\bar{u}, \bar{v})(\phi_i \bar{u}, \phi_i \bar{v})=0$ for $i=1, \dots, k$. \newline

\noindent We will explain the idea in the particular case of the potential $V(x)= (1+ |x|)^2$, that is a single point of minimum in the origin. The idea will be considering functions like:

\begin{equation}
    \bar{u}=t\phi(x)u \left(\frac{x}{\epsilon} \right)+ \bar{\Psi}_t; \qquad \bar{v}=t\phi(x) v\left( \frac{x}{\epsilon}  \right)- \bar{\Psi}_t
\end{equation}

\noindent where $t >0$ and $\bar{\Psi}_t$ is the minimum point associated to $(t\phi(x)u(x/\epsilon), t\phi(x)v (x/ \epsilon))$ as seen in Lemma \ref{lemma2}. With a rescaling like this we will see that when $\epsilon$ is small we can approximate $J_\epsilon '(\bar{u}, \bar{v})(\bar{u},\bar{v})$ with $\epsilon^3 I_{V(0)} '(\bar{u}, \bar{v})(\bar{u},\bar{v})$ with an error that is $o_\epsilon (\epsilon^3)$. With this notation we mean that a quantity $e_\epsilon $ that depends on $\epsilon$ belongs to $o_\epsilon (\epsilon^3)$ if it holds that

\begin{equation}
    \lim_{\epsilon \to 0} \frac{e_\epsilon}{\epsilon^3} =0.
\end{equation}

\noindent Finally we use already proved results on $I_{V(0)} '$ to conclude.

\noindent Fix points $x_i \in \Lambda_i$ such that $V(x_i)= \inf_{\Lambda_i} V$ and consider pairs of positive ground state solutions $u_i, v_i \in H^1 (\mathbb{R}^3)$ of the system

\begin{equation}
\begin{cases}\label{system with V(x_i)}
    -\Delta u_i + V(x_i) u_i = h_v (u_i, v_i)\\
    -\Delta v_i + V(x_i) v_i = h_u (u_i, v_i)
\end{cases}
\end{equation}

\noindent corresponding to the energy level $c_i\coloneqq c(V(x_i))$, and the functional

\begin{equation}
    I_\lambda (u,v) \coloneqq \int \langle \nabla u, \nabla v \rangle + \lambda uv - \int h(u,v).
\end{equation}

\noindent We recall that $c_i = I_{V(x_i)} (u_i, v_i)$ , and following the idea we have given earlier we define

\begin{equation}\label{definition of u_i, epsilon, v_i, epsilon}
u_{i, \epsilon} \coloneqq \phi_i (x) u_i \left(\frac{x-x_i}{\epsilon}\right), \; v_{i, \epsilon} \coloneqq \phi_i (x) v_i\left(\frac{x-x_i}{\epsilon}\right).
\end{equation}

\begin{proposition}\label{prop3.1}
There exists $\epsilon_0 >0$ such that for every $0 < \epsilon < \epsilon_0$ and every $i=1, \dots, k$ there is $\Psi_\epsilon \in H$ and points $t_{1,\epsilon}, \dots , t_{k, \epsilon} \in [1- \mu , 1+\mu]$ for some $\mu >0$ such that the functions:

\begin{equation}\label{form of u_epsilon, v_epsilon}
\bar{u}_\epsilon \coloneqq \sum_{i=1}^k t_{i, \epsilon} u_{i, \epsilon} + \Psi_\epsilon \text{ and } \bar{v}_\epsilon \coloneqq \sum_{i=1}^k t_{i, \epsilon} v_{i, \epsilon} - \Psi_\epsilon
\end{equation}
satisfy:
\begin{align}
    &J_\epsilon ' (\bar{u}_\epsilon, \bar{v}_\epsilon)(\phi_i \bar{u}_\epsilon, \phi_i \bar{v}_\epsilon)=0, \quad \forall i= 1, \dots, k\label{eq3.1.1};\\
    &J_\epsilon '(\bar{u}_\epsilon, \bar{v}_\epsilon)(\phi, -\phi)=0, \quad \forall i =1, \dots, k \; \forall \phi \in H\label{eq3.1.2};\\
    &J_\epsilon (\bar{u}_\epsilon, \bar{v}_\epsilon)= \epsilon ^3 \left( \sum_{i=1}^k c_i + o_\epsilon (1) \right) \text{ as } \epsilon \to 0, \label{eq:3.1.3}
\end{align}
with $u_{i, \epsilon}$ ,$v_{i, \epsilon}$ defined in (\ref{definition of u_i, epsilon, v_i, epsilon}). Moreover
\begin{equation}\label{u_epsilon and v_epsilon weight epsilon cube}
\int_{\Lambda_i} (\bar{u}_\epsilon ^2 +  \bar{v}_\epsilon ^2) \geq \eta \epsilon^3 \; \forall i=1, \dots, k
\end{equation}
for some $\eta >0$.
\end{proposition}

\begin{proof}

We refer the reader to \cite{ramos2008solutions} for a detailed proof, here I will just give the idea. For every given $\epsilon$ and $\bar{t}=(t_1, \dots, t_k) \in [0, 2] \times \dots \times [0,2]$, let:

\begin{equation} 
    u_{\epsilon, \bar{t}} \coloneqq \sum_i t_i u_{i, \epsilon}, \quad v_{\epsilon, \bar{t}} \coloneqq \sum_i t_i v_{i, \epsilon}
\end{equation}
and let $\Psi_{\epsilon, \bar{t}}$ be defined as in Lemma \ref{lemma2}, that is the (unique) minimum point of $\phi \mapsto -J_\epsilon ((u_{\epsilon, \bar{t}}, v_{\epsilon, \bar{t}})+(\phi, -\phi))$, so that

\begin{equation}
    J_\epsilon ' ((u_{\epsilon, \bar{t}}, v_{\epsilon, \bar{t}})+(\Psi_{\epsilon, \bar{t}},-\Psi_{\epsilon, \bar{t}}))(\phi, -\phi)=0 \quad \forall \phi \in H.
\end{equation}
In a similar way we will call $\Psi_{ t_i}$ the minimum point of 

\begin{equation}
    \phi \mapsto -I_{ V(x_i)}((t_i u_{i},t_i v_{i})+ (\phi, -\phi)).
\end{equation}

\noindent We choose $i \in \{1, \dots, k\}$ and we define the following rescaled functions:

\begin{align}
\begin{cases}
    &\Psi_{\bar{t}}^\epsilon (x) \coloneqq \Psi_{\epsilon, \bar{t}} (x_i + \epsilon x);\\
    &\Psi_{t_i}^\epsilon \coloneqq \Psi_{t_i}(x_i + \epsilon x);\\
    &\phi_i ^\epsilon (x) \coloneqq \phi_i (x_i + \epsilon x);\\
    &V_i ^\epsilon (x) \coloneqq V(x_i + \epsilon x);\\
    &h_u ^{\epsilon,i}(x, (s,t)) \coloneqq h_u(x_i + \epsilon x, (s,t));\\
    &h_v ^{\epsilon,i}(x, (s,t))\coloneqq  h_v(x_i + \epsilon x, (s,t)).
    \end{cases}
\end{align}
We will also consider rescaled neighborhoods of $x_1, \dots, x_k$, defined as

\begin{equation}
    \Lambda_i ^\epsilon \coloneqq \frac{\Lambda_i - x_i }{\epsilon} .
\end{equation}
As we have already mentioned earlier, the idea of the proof will be that of substituting $J_\epsilon '$ with $I_{V(x_i)} '$. For this reason, we define the following functions:

\begin{align}
&\theta_{i, \epsilon}(\bar{t}) \coloneqq J_\epsilon ' (( u_{\epsilon, \bar{t}},  v_{\epsilon, \bar{t}}) + ( \Psi_{\epsilon, \bar{t}} ,-\Psi_{\epsilon, \bar{t}}))(u_{\epsilon, \bar{t}}\phi_i, v_{\epsilon, \bar{t}}\phi_i);\\
&\theta_i (t_i) \coloneqq I' _{V(x_i)} (t_i (u_i, v_i) + (\Psi_{t_i},-\Psi_{t_i}))(u_i, v_i).
\end{align} 
We will substitute $\theta_{i, \epsilon }(\bar{t})$ with $\theta_i (t_i)$, or to be more precise we will show that 

\begin{equation}\label{theta_i epsilon substitution theta_i}
\epsilon^{-3} \theta_{i, \epsilon}(\bar{t})=t_i   \theta_i (t_i)+ o_\epsilon (1).
\end{equation}
we will do this by localizing our functionals in $\Lambda_i ^\epsilon$ only, proving that the part outside of $\Lambda_i ^\epsilon$ is $o_\epsilon (1)$, and showing that as $\epsilon \to 0$:

\begin{align}
    \epsilon^{-3} \theta_{i,\epsilon} (\bar{t})=& I'(t_i(u_i,v_i)+ (\Psi_{\bar{t}}^\epsilon, -\Psi_{\bar{t}}^\epsilon))(t_i u_i , t_i v_i))|_{\Lambda_i ^\epsilon} +o_\epsilon (1) \xrightarrow{ \Psi_{\bar{t}}^\epsilon \to \Psi_{t_i}}\nonumber\\
    &  I_{V(x_i)}'(t_i(u_i,v_i) + (\Psi_{t_i}, -\Psi_{t_i}))(t_i u_i, t_i v_i)|_{\Lambda_i ^\epsilon} + o_\epsilon (1)= t_i  \theta_i (t_i) + o_\epsilon(1).
\end{align} 
After proving that both $\theta_i (t_i)$ and $\theta_{i, \epsilon } (\bar{t})$ are $o_\epsilon (\epsilon ^3)$ outside of $\Lambda_i ^\epsilon$, the key observation is that $\Psi_{t_i}$ and $\Psi_{\bar{t}} ^\epsilon$ satisfy almost the same equation in $\Lambda_i ^\epsilon$ (see \cite{ramos2008solutions}). We can conclude with a theorem commonly known as Poincaré-Miranda fixed point theorem (see \cite{miranda1940osservazione}).
\end{proof}

\begin{proposition}\label{||(u,v)||_epsilon estimate proposition}

For every $C_0 >0$ there exist $\epsilon_0 , D_0 , \eta_0 >0$ such that for every $0 < \epsilon < \epsilon_0$:

\begin{equation}
    (u,v) \in N_\epsilon , \; J_\epsilon (u,v) \leq C_0 \epsilon^3 \implies ||(u,v)||_{\epsilon} ^2 \leq D_0 \epsilon^3 \text{ and } \int_{\Lambda_i} (u^2 + v^2) \geq \eta_0 \epsilon^3
\end{equation}
for every $i=1, \dots, k$. It also holds

\begin{equation}
    \epsilon^2 \int_{\Lambda_i} (|\nabla u|^2 + |\nabla v|^2) \geq \eta_1 \epsilon^3 \text{ for some } \eta_1 >0 \; \forall i=1, \dots, k
\end{equation} 
and similarly $J_\epsilon (u,v) \geq \eta_2 \epsilon^3$ and $\int_{\Lambda_i} \langle \nabla h (u,v), (u,v) \rangle \geq \eta_3 \epsilon^3$ for every $i=1,\dots, k$ and some $\eta_2, \eta_3 >0$.
\end{proposition}

\begin{proof}
Using properties h1) -h7) and hm), we can come up with the following inequality:

\begin{align}
||(u,v)||_\epsilon ^2 &= \langle u, \Phi u \rangle_\epsilon+ \langle v, \Phi v \rangle_\epsilon +\langle v-u, \xi (v-u)\rangle_\epsilon + 2\langle u,v \rangle_\epsilon- \langle v, \Phi u \rangle_\epsilon - \langle u, \Phi v \rangle_\epsilon \lesssim \nonumber\\
& (\mu + \delta + o_\epsilon(1))||(u,v)||_\epsilon ^2 + C_\mu (\delta + o_\epsilon (1)) ||(u,v)||_\epsilon ^2 + \epsilon^3,
\end{align}
where $\delta$ comes from condition hm) and $\mu, C_\mu$ from h6). Then by taking $\mu, \delta$ sufficiently small we can get $||(u,v)||_\epsilon ^2 \lesssim \epsilon ^3$. Now using Sobolev inequality we have 

\begin{equation}\label{u^5 v estimate}
\int_{\Lambda_i} (|u|^6 + |v|^6) \leq C \left( \int (|\nabla u|^2 + |\nabla v|^2 ) \phi_i \right)^{3} + \tilde{\delta} \int V(x) (u^2 + v^2)\phi_i,
\end{equation}
with $\tilde{\delta} \to 0$ as $\epsilon \to 0$, indeed we know $||(u,v)||_\epsilon ^2 \lesssim \epsilon ^3$ and thus $\int V(x) (u^2 + v^2)\phi_i <<1$. The assumption $\int_{\Lambda_i} (u^2 +v^2) \geq \epsilon ^4$ together with $||(u,v)||_\epsilon ^2 \leq D_0 \epsilon^3$ implies 

\begin{equation}
    C'\epsilon ^2 \int u^2 \leq D_0 ' \epsilon^5 \leq \frac{1}{4} \int V(x)(u^2 + v^2)\phi_i
\end{equation}
if $\epsilon$ is sufficiently small. Doing the same for $v$ we end up with

\begin{equation}\label{u,v estimate 1/2}
\langle u, \phi_i u \rangle_\epsilon + \langle v, \phi_i v \rangle_\epsilon \geq \frac{1}{2} \left( \epsilon ^2 \int (|\nabla u|^2 + |\nabla v|^2) \phi_i + \int V(x)(u^2 + v^2) \phi_i \right).
\end{equation}
At this point, exploiting again properties h1)-h7), hm) and \eqref{u,v estimate 1/2} we can prove that

\begin{align}\label{gradient exponent estimate}
    &\epsilon^2 \int (|\nabla u|^2 + |\nabla v|^2)\phi_i + \int V(x) (u^2 + v^2) \phi_i \leq 2 \langle u, \phi_i u \rangle_\epsilon + 2\langle v, \phi_i v \rangle_\epsilon = \nonumber \\
    &2 \int \langle \nabla h (x, (u,v)), (v,u) \rangle \phi_i \leq  
    2\mu \int (u^2 + v^2) \phi_i + 2C_\mu \int (|u|^6 + |v|^6)\phi_i \leq \nonumber \\
    &\left( \frac{2\mu}{\alpha} + 2C_\mu \tilde{\delta}\right) \int V(x) (u^2 + v^2) \phi_i + 2C C_\mu \left(\int (|\nabla u|^2 + |\nabla v|^2 )\phi_i\right)^3.
\end{align}
Choosing $\mu, C_\mu$ accordingly, we can prove the estimates on $\int_{\Lambda_i} (u^2 + v^2)$ and $\int_{\Lambda_i} (|\nabla u|^2 + |\nabla v|^2)$. As for the estimate on $J_\epsilon$, we use that

\begin{align}
2J_\epsilon (u,v)&=2 \langle u, v \rangle _\epsilon - 2\int h(x,(u,v))\geq ||u||_\epsilon ^2 + ||v||_\epsilon ^2 + \int_\Lambda \langle \overbrace{\nabla h (u,v), (u,v) \rangle - 2 h (u,v)}^{\geq 0} - \nonumber \\
&-\int_\Lambda \langle \nabla h (x, (u,v)), (v,u) \rangle - \delta ||(u,v)||_{L^2 (\mathbb{R}^3 \setminus \Lambda)}.
\end{align}
Indeed consider $\Phi \coloneqq \sum_{i=1} ^k \phi_i$, then we can use the estimate

\begin{align}
\int \langle \nabla h(x,(u,v)), (v,u) \rangle \Phi &= \langle u, \Phi u \rangle + \langle v, \Phi v\rangle \lesssim \nonumber\\ 
&\lesssim \mu ||(u,v)||_\epsilon ^2 + C_\mu (\delta + o_\epsilon (1))||(u,v)||_\epsilon ^2 + C_\mu  \epsilon ^3 ,
\end{align}
which can be rewritten as 

\begin{equation}
    J_\epsilon (u,v) \geq \eta_2 \epsilon ^3
\end{equation} 
for some $\eta_2 >0$. Finally, the last estimate comes from

\begin{align}
    &C_\mu \int \langle \nabla h(u,v) , (u,v) \rangle \phi_i + \mu \int (u^2 + v^2) \phi_i \geq \int(h_u (u,v)v + h_v(u,v)u) \phi_i = \nonumber \\
    &\langle u, \phi_i u\rangle_\epsilon + \langle v_i, \phi_i v\rangle_\epsilon \geq \frac{1}{2} \left( \epsilon ^2 \int (|\nabla u|^2 + |\nabla v|^2) \phi_i + \int V(x)(u^2 + v^2) \phi_i \right) \geq C \epsilon^3,
\end{align}
by considering a sufficiently small $\mu$. 
\end{proof}

\begin{proposition}\label{negativity of J prime plus J second}
Under the conditions of proposition \ref{||(u,v)||_epsilon estimate proposition}, there exists $\epsilon_0 >0$ such that for any $0 < \epsilon \leq \epsilon_0$, any $\psi \in H$, and any $i \in \{1,\dots, k\}$,

\begin{equation}
    J_\epsilon ' (u,v)(\phi_i ^2 u, \phi_i ^2 v)+ J_\epsilon ''(u,v)(\phi_i u + \psi, \phi_i v - \psi)(\phi_i u +  \psi, \phi_i v - \psi) <0.
\end{equation}
\end{proposition}

\begin{proof}
See \cite{ramos2008solutions}. The proof is based on three technical lemmas.
\end{proof}

\begin{proposition}\label{N_epsilon is submanifold}
The set $N_\epsilon$ is a weakly closed $C^1$ submanifold of $H \times H$.
\end{proposition}
\begin{proof}

Consider $N_\epsilon$ as the zero locus of $K_\epsilon (u,v)=0$, with $K_\epsilon : H\times H \to \mathbb{R}^k \times H^-$ defined as

\begin{equation}
    K_\epsilon (u,v) \coloneqq (J_\epsilon ' (u,v) (\phi_1 u, \phi_1 v), \dots, J_\epsilon ' (u,v) (\phi_k u, \phi_k v), PJ_\epsilon '(u,v))
\end{equation}
where $P: H\times H \to H^- = \{ (\phi, -\phi)| \phi \in H\}$ is the orthogonal projection. It can be proved, exploiting Proposition \ref{negativity of J prime plus J second}, that it is a $C^1$ submanifold thanks to the implicit function theorem. The weak closure can also be easily checked.
\end{proof}

\end{section}

\begin{section}{\textbf{PROOF OF THE MAIN THEOREM}}

We recall that 

\begin{equation}
    c_\epsilon \coloneqq \inf_{N_\epsilon} J_\epsilon
\end{equation}
and we denote

\begin{equation}\label{definition of J_epsilon ^i}
J_\epsilon ^i (u,v) \coloneqq \int_{\tilde{\Lambda}_i} \{\epsilon^2 \langle \nabla u, \nabla v \rangle + V(x) uv \} - h(x,(u,v)).
\end{equation}
We are now ready to prove that there exists a minimum point $(u_\epsilon, v_\epsilon)$ for $J_\epsilon$ restricted to $N_\epsilon$ that actually solves $J_\epsilon ' (u_\epsilon,v_\epsilon)=0$.

\begin{theorem}\label{u epsilon v epsilon is critic point and ground state estimate}
    For every sufficiently small $\epsilon >0$ there exists $(u_\epsilon, v_\epsilon) \in N_\epsilon$ such that
    
    \begin{equation}
        J_\epsilon (u_\epsilon, v_\epsilon)= c_\epsilon \text{ and } J_\epsilon ' (u_\epsilon, v_\epsilon) =0.
    \end{equation} 
    Moreover,
    
    \begin{equation}
        J_\epsilon ^i (u_\epsilon, v_\epsilon)= \epsilon ^{3}(c_i + o_ \epsilon (1)) \; \forall i \text{ and } c_\epsilon = \epsilon ^3 \left( \sum_{i=1}^k c_i + o_\epsilon (1) \right) \text{ as } \epsilon \to 0.
    \end{equation}
\end{theorem}

\begin{proof}
We consider a constrained Palais-Smale sequence for $J_\epsilon$ in $N_\epsilon$ at level $c_\epsilon$, that is there exist (thanks to Lagrange multipliers) $\lambda_i ^n \in \mathbb{R}$ with $i=1, \dots, k$ and $\psi_n \in H$ such that for every $(\zeta , \xi) \in H \times H$ we have:

\begin{equation}\label{lagrange multiplier estimate}
J_\epsilon ' (u_n, v_n)(\zeta, \xi)= J_\epsilon ' (u_n, v_n)(\Phi_n \zeta, \Phi_n \xi) + J_\epsilon '' (u_n, v_n) (\Phi_n u_n + \psi_n , \Phi_n v_n - \psi_n )( \zeta, \xi) + o_n (1)
\end{equation}
where $o_n (1) \to 0$ uniformly for bounded $\zeta, \xi$ and $\Phi_n \coloneqq \sum_i \lambda_i ^n \phi_i$.  The existence of the sequence $(u_n, v_n)$ can be derived from Ekeland variational principle (see \cite{struwe2000variational}).
 
 
 
 
 \noindent Let $\lambda_n \coloneqq (\sum_i (\lambda_i ^n)^2)^{1/2}$, we claim that

\begin{equation}\label{lambda_n to zero psi_n to zero} \lambda_n \to 0 \text{ and } \psi_n \to 0.
\end{equation}
To make computations easier, suppose $k=1$, so that there is only one $\lambda_n$ and only one cutoff function $\phi$. The general case can be done in the same way. Considering $\zeta= (\lambda_n \phi u_n + \psi_n)/\lambda_n$ and $\xi= (\lambda_n \phi v_n - \psi_n)/\lambda_n$, so that $J_\epsilon ' (u_n, v_n)(\zeta, \xi)=0$, we get:

\begin{equation}
    \frac{\eta \epsilon ^3}{2} \lambda_n ^2 + ||\psi_n||_\epsilon ^2 \leq o_\epsilon (1)\epsilon ^{3/2} |\lambda_n| ||\psi_n||_\epsilon + o_n(1) |\lambda_n|
\end{equation} 
where it is necessary to use an upper bound contained in the proof of Proposition \ref{negativity of J prime plus J second} (see \cite{ramos2008solutions}). Using the weak closure of $N_\epsilon$ and Fatou lemma, we can easily prove that $J_\epsilon (u_\epsilon, v_\epsilon)=c_\epsilon$, that is $(u_\epsilon, v_\epsilon)$ is a constrained minimum point for $J_\epsilon$. As for the energy estimate, we can do that by noting that $J_\epsilon (u_\epsilon , v_\epsilon)|_{\mathbb{R}^3 \setminus \Lambda}= o_\epsilon (\epsilon ^3)$, meaning that

\begin{equation}
    \lim_{\epsilon \to 0} \frac{J_\epsilon (u_\epsilon , v_\epsilon)|_{\mathbb{R}^3 \setminus \Lambda}}{\epsilon ^3} \to 0 ,
\end{equation} 
so

\begin{equation}
    J_\epsilon (u_\epsilon , v_\epsilon) = \sum_{i=1}^k J_\epsilon (u_\epsilon \phi _i, v_\epsilon, \phi_i) + o_\epsilon (\epsilon ^3),
\end{equation}
and thanks to Proposition \ref{prop3.1} we know that $c_\epsilon \leq \epsilon ^3 \left(\sum_{i=1}^k c_i + o_\epsilon (1) \right)$. The other inequality, and the estimates on $J_\epsilon ^i (u_\epsilon, v_\epsilon)$ rely on an approximation of $J_\epsilon ^i$ with $\epsilon ^3 I_{V(x_i)} (u_i ^\epsilon,, v_i ^\epsilon)$, with $u_i ^\epsilon  \coloneqq u_\epsilon (x_i + \epsilon x)$ and $v_i ^\epsilon (x)= v_\epsilon (x_i + \epsilon x)$. Using Lemma \ref{lemma6} and Remark \ref{remark Fatou ground state} which tells us that $I_{V(x_i)} (u_i ^\epsilon, v_i ^\epsilon ) \geq c_i + o_\epsilon (1)$ we can conclude.
\end{proof}

\begin{theorem}\label{limsup liminf theorem}

The couple $(u_\epsilon, v_\epsilon)$ satisfy (\ref{main system}) and moreover it holds that

\begin{equation}
    \lim_{\epsilon \to 0} \sup_{\mathbb{R}^3 \setminus \Lambda} \{u_\epsilon, v_\epsilon\} =0 \text{ and } \liminf_{\epsilon \to 0} \min \{ \sup_{\Lambda_i} u_\epsilon, \sup_{\Lambda_i} v_\epsilon \} >0
\end{equation}
for every $i=1 \dots, k$.
\end{theorem}

\begin{proof}

The proof is very technical, and we remind the interested reader to \cite{ramos2008solutions} for the details. The first step consists in showing that 

\begin{equation}
    \sup_{\partial \Lambda_i} \{u_\epsilon, v_\epsilon\} \to 0 \text{ as } \epsilon \to 0.
\end{equation}
Then thanks to the maximum principle we can see that, outside of $\Lambda$, $u_\epsilon$ and $v_\epsilon$ tend to 0. This also means that for a sufficiently small $\epsilon$, the modification on $h$ never occurs and $(u_\epsilon, v_\epsilon)$ satisfy the initial system. To prove that inside $\Lambda_i$ there is a "spike", we proceed by absurd. If we suppose that, for instance, $\sup_{\Lambda_i} u_\epsilon \to 0$, we can use a result by Lions (see Theorem 1.21 in \cite{willem1997minimax}) to prove that this would imply 

\begin{equation}
    \liminf_{\epsilon \to 0} \int_{\Lambda_i ^\epsilon} (u^\epsilon)^2 + (v^\epsilon)^2=\liminf_{\epsilon \to 0} \frac{1}{\epsilon^3} \int_{\Lambda_i} (u^2 + v^2) =0
\end{equation}
which is absurd thanks to what we have proved in Proposition \ref{||(u,v)||_epsilon estimate proposition}.
 \end{proof}
 
\begin{theorem}

The functions $u_\epsilon$, $v_\epsilon$ satisfy the properties (i)-(iv) of Theorem \ref{main theorem}.
\end{theorem}

\begin{proof}

It follows from the previous Theorem that for every $i=1, \dots, k$, which we now fix, there exist $x_{i, \epsilon}, y_{i, \epsilon} \in \Lambda_i$ such that 

\begin{equation}
    u_\epsilon (x_{i, \epsilon})= \max_{\Lambda_i} u_\epsilon \text{ and } v_\epsilon (y_{i, \epsilon})=\max_{\Lambda_i} v_\epsilon.
\end{equation}
We also know that there exists $b>0$ such that

\begin{equation}
    u_\epsilon (x_{i, \epsilon}), v_\epsilon (y_{i, \epsilon}) \geq b >0 \text{ and } \lim V(x_{i, \epsilon})= \lim V(y_{i, \epsilon})= \inf_{\Lambda_i} V \text{ as }\epsilon \to 0.
\end{equation}
We can show that if we define $\xi_{i, \epsilon} \coloneqq (x_{i, \epsilon} - y_{i, \epsilon})/\epsilon$, it holds that $\xi_{i, \epsilon}=0$ for sufficiently small $\epsilon$. Indeed, let

\begin{align}
&\bar{u}_\epsilon \coloneqq u_\epsilon (\epsilon x + x_{i,\epsilon}); \quad \bar{v}_\epsilon \coloneqq v_\epsilon (\epsilon x + x_{i, \epsilon});\\
&\tilde{u}_\epsilon \coloneqq u_\epsilon (\epsilon x + y_{i,\epsilon}); \quad \tilde{v}_\epsilon \coloneqq  v_\epsilon (\epsilon x + y_{i, \epsilon}).
\end{align} 
If it were that $|\xi_{i, \epsilon}| \to + \infty$, then we would have

\begin{equation}
    c_i \geq \liminf_{\epsilon \to 0} \frac{1}{\epsilon^3} J_\epsilon ^i (\bar{u}_\epsilon, \bar{v}_\epsilon) \geq \liminf_{\epsilon \to 0} \frac{1}{\epsilon ^3} (J_\epsilon  (\bar{u}_\epsilon, \bar{v}_\epsilon)|_{B_R} +J_\epsilon (\tilde{u}_\epsilon, \tilde{v}_\epsilon)|_{B_R}) \geq \frac{3}{2} c_i
\end{equation}
for a sufficiently large $R>0$ and a sufficiently small $\epsilon$. Then $|\xi_{i, \epsilon}|$ is bounded and we can prove $\xi_{i, \epsilon} \to 0$ using an argument similar to that used in \cite{ramos2005spike}, that is we consider the limit functions $(\bar{u}, \bar{v})$ for $(\bar{u}_\epsilon, \bar{v}_\epsilon)$ and $(\tilde{u}_\epsilon, \tilde{v}_\epsilon)$ and exploiting the radiality of $(\bar{u}, \bar{v})$ we can see that $x_{i, \epsilon}= y_{i, \epsilon}$ for sufficiently small $\epsilon$. Properties related do the exponential decay come from a standard application of the maximum principle.
\end{proof}

\end{section}

\section{\textbf{ACKNOWLEDGEMENTS}}

The second  author was supported in part by  INDAM, GNAMPA - Gruppo Nazionale per l'Analisi Matematica, la Probabilit\`a e le loro Applicazioni, by Institute of Mathematics and Informatics, Bulgarian Academy of Sciences, by Top Global University Project, Waseda University,  the Project PRA 2018 49 of  University of Pisa and project "Dinamica di equazioni nonlineari dispersive", "Fondazione di Sardegna", 2016.

\bibliographystyle{plain}
\bibliography{Bibliografia.bib}

\begin{thebibliography}{10}

\bibitem{ambrosio2019lectures}
Luigi Ambrosio, Alessandro Carlotto, and Annalisa Massaccesi.
\newblock {\em Lectures on Elliptic Partial Differential Equations}, volume~18.
\newblock Springer, 2019.

\bibitem{NehariMethodThesis}
Tommaso Cortopassi.
\newblock A generalised {N}ehari manifold method for a class of non linear {S}chr\"odinger systems in $\mathbb{R}^3$.
\newblock Master thesis, University of Pisa, June 2021.
\newblock Available at \url{https://etd.adm.unipi.it/t/etd-05282021-110619}.

\bibitem{evans1998partial}
Lawrence~C Evans.
\newblock Partial differential equations.
\newblock {\em Graduate studies in mathematics}, 19(2), 1998.

\bibitem{gilbarg2015elliptic}
David Gilbarg and Neil~S Trudinger.
\newblock {\em Elliptic partial differential equations of second order}, volume 224.
\newblock springer, 2015.

\bibitem{miranda1940osservazione}
Carlo Miranda.
\newblock {\em Un'osservazione su un teorema di Brouwer}.
\newblock Consiglio Nazionale delle Ricerche, 1940.

\bibitem{ramos2006concentration}
Miguel Ramos and S{\'e}rgio~HM Soares.
\newblock On the concentration of solutions of singularly perturbed {H}amiltonian systems in $\mathbb{R}^n$.
\newblock {\em PORTUGALIAE MATHEMATICA}, 63(2):157, 2006.

\bibitem{ramos2008solutions}
Miguel Ramos and Hugo Tavares.
\newblock Solutions with multiple spike patterns for an elliptic system.
\newblock {\em Calculus of Variations and Partial Differential Equations}, 31(1):1--25, 2008.

\bibitem{ramos2005spike}
Miguel Ramos and Jianfu Yang.
\newblock Spike-layered solutions for an elliptic system with {N}eumann boundary conditions.
\newblock {\em Transactions of the American Mathematical Society}, 357(8):3265--3284, 2005.

\bibitem{shen1995p}
Zhongwei Shen.
\newblock Lp estimates for {S}chr{\"o}dinger operators with certain potentials.
\newblock In {\em Annales de l'institut Fourier}, volume~45, pages 513--546, 1995.

\bibitem{struwe2000variational}
Michael Struwe and M~Struwe.
\newblock {\em Variational methods}, volume 991.
\newblock Springer, 2000.

\bibitem{willem1997minimax}
Michel Willem.
\newblock {\em Minimax theorems}, volume~24.
\newblock Springer Science \& Business Media, 1997.

\end{thebibliography}

\end{document}